\documentclass{amsart}

\usepackage{amssymb}
\usepackage[hidelinks]{hyperref}

\theoremstyle{plain}
\newtheorem{theorem}{Theorem}[section]
\newtheorem{lemma}[theorem]{Lemma}
\newtheorem{corollary}[theorem]{Corollary}
\newtheorem{proposition}[theorem]{Proposition}

\theoremstyle{definition}

\newtheorem{example}[theorem]{Example}

\theoremstyle{remark}

\newcommand{\bN}{\mathbb{N}}
\newcommand{\N}{\bN}

\newcommand{\cA}{\mathcal{A}}
\newcommand{\cB}{\mathcal{B}}

\newcommand{\cF}{\mathcal{F}}

\newcommand{\cH}{\mathcal{H}}
\newcommand{\cI}{\mathcal{I}}
\newcommand{\I}{\cI}
\newcommand{\cJ}{\mathcal{J}}
\newcommand{\J}{\cJ}

\newcommand{\cN}{\mathcal{N}}
\newcommand{\cP}{\mathcal{P}}

\newcommand{\cW}{\mathcal{W}}

\newcommand{\Hindman}{\mathcal{H}} 
\newcommand{\Folkman}{\mathcal{F}} 
\DeclareMathOperator{\FS}{FS}

\begin{document}


\title{New Hindman spaces}


\address{Institute of Mathematics\\ Faculty of Mathematics, Physics and Informatics\\ University of Gda\'{n}sk\\ ul.~Wita Stwosza 57\\ 80-308 Gda\'{n}sk\\ Poland}

\author[R.Filip\'{o}w]{Rafa\l{} Filip\'{o}w}
\email{Rafal.Filipow@ug.edu.pl}
\urladdr{http://mat.ug.edu.pl/~rfilipow}

\author[K.~Kowitz]{Krzysztof Kowitz}
\email{Krzysztof.Kowitz@phdstud.ug.edu.pl}

\author[A.~Kwela]{Adam Kwela}
\email{Adam.Kwela@ug.edu.pl}
\urladdr{http://kwela.strony.ug.edu.pl/}

\author[J.~Tryba]{Jacek Tryba}
\email{Jacek.Tryba@ug.edu.pl}


\date{\today}






\begin{abstract}
We introduce a method that allows to turn topological questions about Hindman spaces into purely combinatorial questions about the Kat\v{e}tov order of ideals on $\N$.
We also provide two applications of the method.
\begin{enumerate}
\item We characterize $F_\sigma$ ideals $\I$ for which there is a Hindman space which is not an $\I$-space under the continuum hypothesis. 
This reduces a topological question of Albin L.~Jones about consistency of existence of a Hindman space which is not van der Waerden to the question whether the ideal of all non AP-sets is not below the ideal of all non IP-sets in the Kat\v{e}tov order.
	\item 
Under the continuum hypothesis, we construct a Hindman space which is not an $\I_{1/n}$-space. 
This answers a question posed by Jana Fla\v{s}kov\'{a} at the 22nd Summer Conference on Topology and its Applications.
\end{enumerate}
\end{abstract}


\maketitle


In our paper we consider topological spaces which satisfy some properties connected with  well known combinatorial theorems, in particular van der Waerden's arithmetical progressions theorem and Hindman's finite sums theorem. Research into this topic was initiated by Kojman and later continued by Shelah and others in \cite{MR1866012,MR1887003,MR1950294,MR2052425,MR2471564,MR3097000,MR3276758}.

A topological space $X$ is called \emph{van der Waerden} \cite{MR1866012} if for every sequence
$\langle x_n\rangle_{n\in \N}$ in $X$  there exists a convergent subsequence 
$\langle x_{n}\rangle_{n\in A}$ with $A$ being an AP-set (i.e.~$A$ contains arithmetic progressions of arbitrary finite length). 

A nonempty family $\I\subseteq\cP(\N)$ of subsets of $\N$ is an \emph{ideal on $\N$} if it is closed under taking subsets and finite unions of its elements, $\N\notin\I$ and $\I$ contains all finite subsets of $\N$. It is easy to see that the family $\I_{1/n}=\{A\subseteq\N: \sum_{n\in A}\frac{1}{n+1}<\infty\}$ 
is an ideal on $\N$, and it follows from van der Waerden's theorem \cite{vanderWearden} that the family 
$\cW = \{A\subseteq\N: \text{$A$ is not an AP-set}\}$ is an ideal on $\N$. An ideal on $\N$ is called \emph{$F_\sigma$} if it is an $F_\sigma$ subset of $\cP(\N)$ with a topology induced from the Cantor space $\{0,1\}^\N$ (recall: a set is $F_\sigma$ if it is a countable union of closed sets).
It is not difficult to check that both $\I_{1/n}$ and $\cW$ are $F_\sigma$.

A topological space $X$ is called an \emph{$\I$-space} \cite{MR2471564} if for every sequence
$\langle x_n\rangle_{n\in \N}$ in $X$  there exists a converging subsequence 
$\langle x_{n}\rangle_{n\in A}$ with $A\notin \I$.
In particular, van der Waerden spaces coincide with $\cW$-spaces.

A set $A\subseteq \N$ is an \emph{IP-set}  if there exists an infinite set $D\subseteq\N$ such that $\FS(D)\subseteq A$ where $\FS(D)$ denotes the set of all finite non-empty sums of distinct elements of $D$.
The family $\Hindman = \{A\subseteq\N: \text{$A$ is not an IP-set}\}$ is an ideal on $\N$ (it follows from Hindman's theorem \cite{MR349574}). 

	In \cite{MR1887003}, Kojman proved that only finite spaces are $\cH$-spaces, and to obtain a meaningful topological counterpart of Hindman's finite sums theorem he utilized  the
notion of IP-convergence introduced by Furstenberg and Weiss in \cite{MR531271}.

An \emph{IP-sequence} in $X$ is a sequence indexed by $\FS(D)$ for some infinite $D\subseteq\N$.
An IP-sequence $\langle x_n\rangle_{n\in \FS(D)}$ in a topological space $X$ \emph{IP-converges} \cite{MR531271} to a point $x\in X$ if for every neighborhood $U$ of the point $x$ there exists $m\in \N$ so that $\{x_n: n\in \FS(D\setminus \{0,1,\dots,m\})\}\subseteq U$ (point $x$ is called the \emph{IP-limit} of the sequence).

A topological space $X$ is called \emph{Hindman} \cite{MR1887003} if for every sequence
$\langle x_n\rangle_{n\in \N}$ in $X$  there exists an infinite set $D\subseteq \N$ such that the subsequence $\langle x_n\rangle_{n\in \FS(D)}$ IP-converges to some $x\in X$.

	In \cite{MR250257}, Kat\v{e}tov introduced the following order to study convergence
in topological spaces.
For ideals $\I,\J$  we say that $\I$ is below $\J$ in the \emph{Kat\v{e}tov order} 
 if there is a function $f:\N\to \N$ such that $f^{-1}[A]\in\J$ for every $A\in \I$. We denote it by $\I\leq_K\J$. When $\I$ is not below $\J$ in Kat\v{e}tov order, we denote it by $\I\not\leq_K\J$.

\smallskip

In 2003, Kojman and Shelah \cite{MR1950294} constructed under the continuum hypothesis a van der Waerden space (i.e.~$\cW$-space) which  is not Hindman. Later Fla\v{s}kov\'{a} \cite{MR2471564} extended their result by constructing an $\I$-space which is not Hindman for every $F_\sigma$ ideal $\I$ under the assumption of Martin's axiom. In particular, she constructed $\I_{1/n}$-space which is not Hindman and posed a question \cite{Flaskova-slides} whether it is consistent that there is a Hindman space which is not an $\I_{1/n}$-space.  In Section~\ref{sec:Hindman-spaces-which-is-not-summable-space} we show that the answer to her question is positive under the continuum hypothesis. This follows from our main result: in Section ~\ref{sec:new-Hindman-spaces} we show that 
if an ideal $\I$ is not below the ideal $\cH$ in the Kat\v{e}tov order
then there exists a Hindman space which is not an $\I$-space. 

\smallskip

The spaces constructed by Kojman, Shelah and Fla\v{s}kov\'{a} are examples of Mr\'{o}wka spaces. In Section~\ref{sec:mrowka-is-never-Hindman} we show that Mr\'{o}wka spaces are not Hindman, therefore the spaces we construct in Section~\ref{sec:new-Hindman-spaces} are not Mr\'{o}wka spaces.

\smallskip

In \cite{MR2052425} Jones posed a question whether there is a Hindman space which is not van der Waerden. 
In fact, we show that, in the realm of $F_\sigma$ ideals, the consistency of existence of a Hindman space which is not an $\I$-space is equivalent to $\I\not\leq_K\cH$.
Thus, the topological question of Jones is reduced to a purely combinatorial question whether $\cW\not\leq_K\cH$.


\section{Mr\'{o}wka spaces are not Hindman}
\label{sec:mrowka-is-never-Hindman}

In the sequel we use $\omega$ to denote the set of all natural numbers, and we identify $n\in \omega$ with the set $\{0,1,\dots,n-1\}$ (in particular $A\setminus n = A\setminus \{0,1,\dots,n-1\}$).

We say that sets $A,B$ are \emph{almost disjoint} if $A\cap B$ is finite.

Let $\cA$ be an infinite pairwise almost disjoint family of infinite subsets of $\omega$.
Define a topological space $\Psi(\cA)$ as follows:
the underlying set of $\Psi(\cA)$ is $\omega\cup \cA$, the points of $\omega$
are isolated and a basic neighborhood of $A\in\cA$ has the form $\{A\}\cup
(A\setminus F)$, with $F$ finite. (The space $\Psi(\cA)$ was introduced in
\cite{MR63650}.)

Let $\Phi(\cA)  = \Psi(\cA)\cup \{\infty\}$ be the one-point compactification of $\Psi(\cA)$. 

If $\cA$ is a \emph{mad family on $\omega$} (i.e. infinite maximal pairwise almost disjoint family of infinite subsets of
$\omega$), then the space $\Phi(\cA)$ 
is called a \emph{Mr\'{o}wka space}.

In \cite{MR1887003}, Kojman proved that if $\cA\subseteq \Hindman$, then the Mr\'{o}wka space $\Phi(\cA)$ is not Hindman, and using this observation various spaces were shown to be not Hindman  (see e.g.~\cite{MR1950294,MR2052425,MR2471564}). Below we show that a Mr\'{o}wka space is not Hindman even without assuming that $\cA\subseteq \Hindman$.

\begin{proposition}
\label{H1}
No Mr\'{o}wka space is Hindman.
\end{proposition}

\begin{proof}
Fix a Mr\'{o}wka space $\Phi(\cA)$ defined by a mad family $\cA$. 

Let 
$P_k =\{n2^{k+1}+2^k: n\in \omega\}$
for each $k\in\omega$ and notice that for any infinite $D\subseteq \omega$ we have:
\begin{enumerate}
	\item 
$\FS(D)\cap P_k$ is infinite for some $k\in \omega$,
\item $\FS(D)\setminus P_k$ is infinite for each $k\in \omega$.
\end{enumerate} 
Indeed, we have two cases: $D\cap P_k$ is infinite for some $k$ or $D\cap P_k$ is finite for each $k$.

In the first case, take $k_0$ such that $D\cap P_{k_0}$ is infinite. 
Then $\FS(D)\cap P_{k_0}$ is infinite as it contains $D\cap P_{k_0}$.
Moreover
for $k\neq k_0$ we have $\FS(D)\setminus P_k$ is infinite as $P_k\cap P_{k_0}=\emptyset$ and $D\cap P_{k_0}\subseteq \FS(D)$, and 
$\FS(D)\setminus P_{k_0}$ is infinite as $a+b\in \FS(D)\setminus P_{k_0}$ for any $a,b\in D\cap P_{k_0}$.

In the second case, take $k_0$ such that $D\cap P_{k_0}\neq\emptyset$. Then 
$\FS(D)\cap P_{k_0}$ is infinite as $a+b\in \FS(D)\cap P_{k_0}$ for any $a\in D\cap P_{k_0}$ and $b\in D\cap P_l$ with $l>k_0$.
Moreover,
$\FS(D)\setminus P_{k}$ is infinite as $D\subseteq \FS(D)$ and $D\setminus P_{k}$ is infinite for any $k$.

Let $A_1,A_2,\dots$ be a sequence of distinct sets from $\cA$.
Let $f:\omega\to \Phi(\cA)$ be a sequence such that $f\restriction P_k:P_k\to A_k\setminus \bigcup_{i<k}A_i$ is a bijection for each $k\in \omega$.
We claim that $f$ has no IP-convergent IP-subsequence.

We take any infinite $D\subseteq\omega$ and show that $f\restriction \FS(D)$ is not IP-convergent to any $x\in \Phi(\cA)$.

If $x\in \omega$, then $f\restriction\FS(D)$ is not IP-convergent to $x$ as $\omega$ is discrete in $\Phi(\cA)$ and $f$ is one-to-one.  

Suppose that $x=\infty$. Take $A\in \cA$ such that $f[D]\cap A$ is infinite (there is one such $A$ since $\cA$ is a mad family) and notice that for each $n\in \omega$ the set $f[D\setminus n]\cap A$ is also infinite.
Then $U=\Phi(\cA) \setminus (\{A\} \cup A)$ is an open neighborhood of $\infty$  
such that $f[\FS(D\setminus n)]\setminus U$ is infinite for each $n$. Hence $f\restriction \FS(D)$ is not IP-convergent to $\infty$.

Suppose that $x\in \cA$. We have two cases: 
$x\neq A_k$ for any $k\in\omega$ or $x = A_{k_0}$ for some $k_0\in \omega$.

In the first case, 
$U=\{x\}\cup x$ is an open neighborhood of $x$ such that $f[\FS(D\setminus n)]\setminus U$ is infinite for each $n$ and therefore $f\restriction \FS(D)$ is not IP-convergent to $x$. 

In the second case, $U=\{A_{k_0}\}\cup (A_{k_0} \setminus \bigcup_{i<k_0} (A_i\cap A_{k_0}))$ is an open neighborhood of $x$
such that $f[\FS(D\setminus n)]\setminus U$ is infinite for each $n$ 
and therefore $f\restriction \FS(D)$ is not IP-convergent to $x$. 
\end{proof}


\section{New Hindman spaces}
\label{sec:new-Hindman-spaces}

In this section we prove our main result that under the continuum hypothesis if $\I\not\leq_K\cH$ then there is a Hindman space which is not an $\I$-space.

First we prove some lemmas we are going to use in the proof of the main theorem.

\begin{proposition}
	\label{prop:IP-limit-unique-iff-limit-unique}
	Let $X$ be a topological space.
	The following conditions are equivalent.
	\begin{enumerate}
		\item Limits of convergent sequences in $X$ are unique.
		\item IP-limits of IP-convergent IP-sequences in $X$ are unique.
	\end{enumerate}
\end{proposition}

\begin{proof}
	$(1)\implies (2)$
	Follows from the fact that if $\langle x_n\rangle_{n\in\FS(D)}$ is an IP-sequence IP-convergent to  $x$ then $\langle x_n\rangle_{n \in D}$ is a sequence convergent to $x$.

	$(2)\implies (1)$
	Suppose that there is a sequence $x_n\in X$ with two distinct limits $x,y\in X$.
	Let 
	$P_k =\{n2^{k+1}+2^k: n\in \omega\}$
	for each $k\in\omega$. 
	Now we define a sequence $y_n\in X$ by
	$y_n = x_k \iff n\in P_k$.
	Let $E=\{2^k:k\in\omega\}$. Then it is not difficult to see that 
	$\langle y_n\rangle_{n\in \FS(E)}$ is IP-convergent to both $x$ and $y$.
\end{proof}

An infinite set $D\subseteq\omega$ is \emph{sparse} if for every $x\in \FS(D)$ there is a unique nonempty finite set $F\subseteq D$ with $x=\sum_{i\in F}i$.
 A sparse set $D$ is \emph{very sparse} if additionally for each $F,F'\in[D]^{<\omega}$, if $F\cap F'\neq\emptyset$ then $\sum_{i\in F}i+\sum_{i\in F'}i\notin \FS(D)$. 
	(The term a ``sparse set'' was coined by Kojman in \cite{MR1887003}, however the sparsity of sets was used earlier as this notion is crucial in Hindman's original proof~\cite{MR349574}. The same notion has been recently renamed ``apartness'' in some other contexts~\cite{MR3731714}.)

\begin{lemma}[{Essentially~\cite[Lemma~2.3]{MR307926}}]
\label{verysparse}
For every infinite $D$ there is an infinite very sparse $D'\subseteq D$.
\end{lemma}

\begin{proof}
Let $D$ be infinite. Inductively pick $d_i\in D$ in such a way that $d_n>2\sum_{i<n}d_i$.
Define $D'=\{d_i:i\in\omega\}$. Obviously, $D'\subseteq D$. Moreover, $D'$ is sparse. Indeed, if 
$x\in \FS(D)$ and $F,F'\in[D]^{<\omega}$ are such that $x=\sum_{i\in F}i=\sum_{i\in F'}i$, then $\sum_{i\in F\setminus F'}i=\sum_{i\in F'\setminus F}i$. Assume to the contrary that $F\neq F'$. Since $(F\setminus F')\cap(F'\setminus F)=\emptyset$, without loss of generality we may assume that $d_n=\max(F\setminus F')>\max(F'\setminus F)$. Observe that $\sum_{i\in F'\setminus F}i\leq \sum_{i<n}d_i<d_n\leq\sum_{i\in F\setminus F'}i$. Thus, $F=F'$.

Finally, we show inductively that $D'$ is very sparse. For the first inductive step note that if $F,F'\subseteq\{d_0\}$ and $F\cap F'\neq\emptyset$ then $F=F'=\{d_0\}$ and $\sum_{i\in F}i+\sum_{i\in F'}i=2d_0<d_1$. Thus, $\sum_{i\in F}i+\sum_{i\in F'}\notin \FS(D')$. Assume now that for some $n\in\omega$, if $F,F'\subseteq\{d_i:i<n\}$ and $F\cap F'\neq\emptyset$ then $\sum_{i\in F}i+\sum_{i\in F'}\notin \FS(D')$. Fix $F,F'\subseteq\{d_i:i\leq n\}$ such that $F\cap F'\neq\emptyset$. 

There are two possibilities: either $\max(F\cup F')<d_n$ and then $\sum_{i\in F}i+\sum_{i\in F'}i\notin \FS(D')$ by the inductive assumption, or $\max(F\cup F')=d_n$.
In the latter case note that $\sum_{i\in F}i+\sum_{i\in F'}i\leq 2\sum_{i\leq n}d_i<d_{n+1}$.
Moreover, if $d_n\in F\cap F'$ then $\sum_{i\in F}i+\sum_{i\in F'}i\geq 2d_n>\sum_{i\leq n}d_i$ and thus $\sum_{i\in F}i+\sum_{i\in F'}i\notin \FS(D')$.
On the other hand, if $d_n\in F\setminus F'$, then suppose to the contrary that $\sum_{i\in F}i+\sum_{i\in F'}i\in \FS(D')$.
In this case, since $\sum_{i\in F}i+\sum_{i\in F'}i<d_{n+1}$,
$\sum_{i\in F}i+\sum_{i\in F'}i\in FS(\{d_i:i\leq n\})=\{d_n\}\cup FS(\{d_i:i<n\})\cup (d_n+FS(\{d_i:i<n\}))$. 
As $d_n>2\sum_{i<n}d_i$ and $d_n\in F\setminus F'$, we get that $\sum_{i\in F}i+\sum_{i\in F'}i\in d_n+FS(\{d_i:i<n\})$ i.e. $\sum_{i\in F\setminus\{d_n\}}i+\sum_{i\in F'}i\in \FS(\{d_i:i<n\})$.
But this contradicts the inductive assumption. Thus, $D'$ is very sparse.
\end{proof}

\begin{lemma}
\label{FS1}
Let $(D_i)\subseteq[\omega]^\omega$ be such that $D_0$ is very sparse and $\FS(D_0)\supseteq \FS(D_1)\supseteq\ldots$. Then there is an infinite $E$ such that $\FS(E)\subseteq \FS(D_0)$ and for each $n\in\omega$ there is $k\in\omega$ with $\FS(E\setminus k)\subseteq \FS(D_n)$.
\end{lemma}

\begin{proof}
Since $D_0$ is very sparse, for each $x\in \FS(D_0)$ there is a unique $F(x)\in[D_0]^{<\omega}$ with $x=\sum_{i\in F(x)}i$. Let $e_0\in D_0$, $e_1\in \FS(D_1)\setminus\{x\in \FS(D_0): F(x)\cap F(e_0)\neq\emptyset\}$ and $e_n\in \FS(D_n)\setminus\{x\in \FS(D_0): (\exists i<n)\, F(x)\cap F(e_i)\neq\emptyset\}$ for all $n>1$. This is possible as $\{x\in \FS(D_0): (\exists i<n)\, F(x)\cap F(e_i)\neq\emptyset\}\in\Hindman \restriction \FS(D_0)$ for each $n\in\omega$.

Define $E=\{e_i:i\in\omega\}$. Then $E$ is infinite. We need to check that $E$ is as needed. 

Let $x\in \FS(E)$. Then there is a finite $F\subseteq \omega$ such that $x=\sum_{i\in F}e_i$. As $F(e_i)$ are pairwise disjoint, we get that $x\in \FS(D_0)$.

Fix $n\in\omega$ and let $x\in \FS(E\setminus e_n)$. Then there is a finite $F\subseteq \omega\setminus n$ such that $x=\sum_{i\in F}e_i$. Recall that each $e_i$ for $i\geq n$ is in $\FS(D_n)$. Thus, for each $i\in F$ we can find finite $T_i\subseteq D_n$ with 
$e_i=\sum_{d\in T_i}d$. Observe that $F(d)\cap F(d')=\emptyset$ for all $d,d'\in D_n$ as otherwise we would get $d+d'\in \FS(D_n)\setminus \FS(D_0)$ which contradicts $\FS(D_0)\supseteq \FS(D_n)$. We claim that the sets $T_i$ are pairwise disjoint. Indeed, if $d\in T_i\cap T_j\subseteq D_n$ for some $i,j\in F$ then $F(d)\subseteq \bigcup_{d'\in T_i}F(d')=F(e_i)$ and $F(d)\subseteq \bigcup_{d'\in T_j}F(d')=F(e_j)$ (as $F(x)$ are unique). This leads to $F(e_i)\cap F(e_j)\neq\emptyset$ which implies $e_i+e_j\notin \FS(D_0)$ and contradicts the fact that $F(e_i)$ are pairwise disjoint. Since the sets $T_i\in[D_n]^{<\omega}$ are pairwise disjoint, $x=\sum_{i\in F}\sum_{d\in T_i}d\in \FS(D_n)$. 
\end{proof}

\begin{lemma}
\label{FS}
Let $D\subseteq\omega$ be infinite and $A_n \in \Hindman$ for each $n\in \omega$. Then there is an infinite $D'$ such that $\FS(D')\subseteq \FS(D)$ and for each $n\in\omega$ there is $k\in\omega$ with $\FS(D'\setminus k)\cap A_n=\emptyset$.
\end{lemma}

\begin{proof}
By Lemma~\ref{verysparse}, without loss of generality we may assume that $D$ is very sparse. Since $A_0\in\Hindman$, the set $\FS(D)\setminus A_0$ contains $\FS(D_0)$ for some infinite $D_0\subseteq\omega$. Similarly, the set $\FS(D_0)\setminus A_1$ contains $\FS(D_1)$ for some infinite $D_1\subseteq\omega$. Thus, we can inductively define a sequence of sets $D_n \in [\omega]^\omega$ such that $\FS(D)\supseteq \FS(D_0)\supseteq \FS(D_{1})\supseteq\ldots$ and $\FS(D_i)\cap A_i=\emptyset$ for all $i\in\omega$. Applying Lemma \ref{FS1} we get an infinite $E$ such that for each $n\in\omega$ there is $k\in\omega$ with $\FS(E\setminus k)\subseteq \FS(D_n)$. Let $k_0$ be such that $\FS(E\setminus k_0)\subseteq \FS(D_0)$ and define $D'=E\setminus k_0$. Then $\FS(D')=\FS(E\setminus k_0)\subseteq \FS(D_0)\subseteq \FS(D)$. Moreover, given $n\in\omega$, as there is $k\in\omega$ with $\FS(E\setminus k)\subseteq \FS(D_n)$, we have that $\FS(D'\setminus k)\subseteq \FS(E\setminus k)\subseteq \FS(D_n)$, so $\FS(D'\setminus k)\cap A_n=\emptyset$.
\end{proof}

Finally, we are ready for the main result of this section.
To some extent, the idea of the proof of the following theorem shares some similarity with Kojman-Shelah's construction of van der Waerden space which is not Hindman (\cite{MR1950294}).

\begin{theorem}
\label{THM}
Assume the continuum hypothesis. If $\I\not\leq_K\Hindman$,
then there is a $T_1$  separable Hindman space which is not an $\I$-space.  Moreover, limits of convergent sequences and IP-limits of IP-convergent IP-sequences in this space  are unique.
\end{theorem}

\begin{proof}
Using the  continuum hypothesis, we only need countable ordinals to fix a list $\langle (H_\alpha,f_\alpha):\alpha<\omega_1\rangle$ of all sets $H_\alpha\notin\Hindman$ and functions $f_\alpha:H_\alpha\to\omega$ satisfying $f_\alpha^{-1}[\{n\}]\in\Hindman$ for all $n\in\omega$. 

We will construct a sequence $\langle D_\alpha:\alpha<\omega_1\rangle$ of infinite subsets of $\omega$ such that for every $\alpha<\omega_1$ 
we have $\FS(D_\alpha)\subseteq H_\alpha$
and $f_\alpha[\FS(D_\alpha)]\in\I$
 and 
one of the following conditions holds:
\begin{equation}
\label{eq:W1}
\tag{W1} \forall {\beta<\alpha}  \, \forall {i\in\omega} \, \exists {k\in \omega} \, \left(f_\alpha[\FS(D_\alpha\setminus k) ]\cap (i\cup f_\beta[\FS(D_\beta\setminus k)])=\emptyset\right)
\end{equation}
or
\begin{equation}
\label{eq:W2}
\tag{W2} \exists {\beta<\alpha} \, \forall { k\in\omega} \,  \exists {n\in \omega} \,  \,\left(f_\alpha[\FS(D_\alpha\setminus n) ]\subseteq f_\beta[\FS(D_\beta\setminus k)]\setminus k\right).
\end{equation}

Suppose that $\alpha<\omega_1$ and that $D_\beta$ have been chosen for all $\beta<\alpha$ . Since $\Hindman$ is homogeneous (see \cite[Example~2.6]{MR3594409}) and $\I\not\leq_K\Hindman$, there is $C\in\I$ such that $f^{-1}_\alpha[C]\notin\Hindman \restriction H_\alpha$. 
Find infinite $D\subseteq\omega$ with $\FS(D)\subseteq H_\alpha$ and $f_\alpha[\FS(D)]\subseteq C$. 
By Lemma \ref{verysparse} we can assume that $D$ is very sparse.

We have 2 cases:
\begin{equation}
\label{eq:P1}
\tag{P1}
\forall {D'\in[\omega]^\omega}\, \forall {\beta<\alpha} \, \left(\FS(D')\subseteq \FS(D)  \implies  \exists {k\in\omega}\, \left( \FS(D')\setminus 
f_\alpha^{-1}[
f_\beta[\FS(D_\beta\setminus k)]]\notin \Hindman\right)\right)
\end{equation}
or
\begin{equation}
\label{eq:P2}
\tag{P2}
\exists {D'\in[\omega]^\omega} \, \exists {\beta<\alpha} \, \left(\FS(D')\subseteq \FS(D)  \land \forall {k\in\omega}\, \left( \FS(D')\setminus f_\alpha^{-1}[f_\beta[\FS(D_\beta\setminus k)]]\in \Hindman\right)\right).
\end{equation}

In the first case, let $\alpha\times \omega =\{(\beta_n,i_n):n<\omega\}$, taking into account that $\alpha$ is countable. 
Using condition \eqref{eq:P1} repeatedly  
and the fact that $f_\alpha^{-1}[i]\in\Hindman$ for every $i\in\omega$, 
one can easily construct a sequence $\langle E_n: n\in\omega\rangle$ of infinite subsets of $\omega$ such that 
\begin{enumerate}
	\item $\FS(E_0)\subseteq \FS(D)$,
	\item $\forall {n\in \omega}\, (\FS(E_{n+1})\subseteq \FS(E_{n}))$,
	\item $\forall {n\in \omega} \, \exists {k\in \omega} \left(\FS(E_n) \cap  f_\alpha^{-1}[i_n\cup f_{\beta_n}[\FS(E_{\beta_n}\setminus k)]] = \emptyset\right)$.
\end{enumerate}

Now using Lemma~\ref{FS1} we find an infinite set $E'\subseteq \omega$ such that 
$\FS(E')\subseteq\FS(D)$
and for every  $n\in\omega$ there is $l\in\omega$ with $\FS(E'\setminus l) \subseteq \FS(E_n)$.

It is not difficult to see that $D_\alpha = E'$ satisfies $\FS(D_\alpha)\subseteq H_\alpha$ and  \eqref{eq:W1}.

Consider the second case.
Let $D'\in[\omega]^\omega$ and $\beta<\alpha$ be such that 
$\FS(D')\subseteq \FS(D)$ and $\FS(D')\setminus f_\alpha^{-1}[f_\beta[\FS(D_\beta\setminus k)]]\in \Hindman$ for each $k\in\omega$. 
Since $f_\alpha^{-1}[i]\in\Hindman$ for every $i\in\omega$,
we also have 
$\FS(D')\setminus f_\alpha^{-1}[f_\beta[\FS(D_\beta\setminus k)]\setminus k]\in \Hindman$ for each $k\in\omega$. 
Using Lemma~\ref{FS}, we find an infinite set $D''\subseteq \omega$
such that $\FS(D'')\subseteq \FS(D')$ and for every $k\in\omega$ there is $n\in\omega$ with 
$\FS(D''\setminus n)\cap (\FS(D')\setminus f_\alpha^{-1}[f_\beta[\FS(D_\beta\setminus k)]\setminus k]) = \emptyset $.

It is not difficult to see that $D_\alpha = D''$ satisfies $\FS(D_\alpha)\subseteq H_\alpha$ and  \eqref{eq:W2}.

The construction of sets $D_\alpha$ is finished.

We are ready to define the required space. Let 
$T=\{\alpha<\omega_1:\text{$D_\alpha$ satisfies \eqref{eq:W1}}\}$
 and $$X=\omega\cup\{\FS(D_\alpha):\alpha\in T\}\cup\{\infty\}.$$ 
For every $x\in X$ we define the family $\cB(x)\subseteq \cP(X)$ as follows:
\begin{itemize}
	\item 
$\cB(n)=\left\{\{n\}\right\}$ for $n\in \omega$,
\item $\cB(\FS(D_\alpha))=\left\{\{\FS(D_\alpha)\}\cup f_\alpha[\FS(D_\alpha\setminus k)]\setminus k:k\in\omega\right\}$  for $\alpha\in T$,
\item $\cB(\infty)=\left\{\{\infty\}\cup\bigcup_{\alpha\in T\setminus F}U_\alpha : F\in[T]^{<\omega} \land U_\alpha\in\cB(FS(D_\alpha)) \text{ for $\alpha\in T\setminus F$}\right\}$.
\end{itemize}
It is not difficult to check that 
the family $\cN = \{\mathcal{B}(x):x\in X\}$ 
satisfies:
\begin{itemize}
	\item for each $x\in X$ we have $\mathcal{B}(x)\neq\emptyset$ and if $U\in\mathcal{B}(x)$ then $x\in U$;
	\item for each $x\in X$, if $U,U'\in\mathcal{B}(x)$ then there is $V\in\mathcal{B}(x)$ with $V\subseteq U\cap U'$;
	\item if $x,y\in X$ are such that $y\in U$ for some $U\in\mathcal{B}(x)$ then there is $V\in\mathcal{B}(y)$ with $V\subseteq U$.
\end{itemize}
Hence $\cN$ is a  neighborhood system  (see e.g.~\cite[Proposition 1.2.3]{MR1039321}).
We claim that $X$ with the topology generated by $\cN$  is a topological space we are looking for.

First, observe that $X$ is separable as $\omega$ is a countable dense subset of $X$. 

Second, it is not difficult to see that $X$ is $T_1$. In fact $X\setminus\{\infty\}$ is $T_2$ (i.e.~Hausdorff), however it is unlikely to be able to separate points $x=\infty$ and $y=\FS(D_\alpha)$.

Now we show that $X$ is a Hindman space. Fix any $f:\omega\to X$. If there is $x\in X$ with $f^{-1}[\{x\}]\notin\Hindman$ then find $D\in[\omega]^\omega$ with $FS(D)\subseteq f^{-1}[\{x\}]$ and observe that $\langle f(n)\rangle_{n\in FS(D)}$ IP-converges to $x$. Thus, we can assume that $f^{-1}[\{x\}]\in\Hindman$ for all $x\in X$. There are two possible cases:
$f^{-1}[X\setminus\omega]\notin\Hindman$ or 
$f^{-1}[\omega]\notin\Hindman$.

If $f^{-1}[X\setminus\omega]\notin\Hindman$ then find $D\in[\omega]^\omega$ with $FS(D)\subseteq f^{-1}[X\setminus\omega]$. As $f^{-1}[\{x\}]\in\Hindman$ for all $x\in X$
and 
$f^{-1}[\{x\}] \neq\emptyset$ only for countably many $x\in X$, using Lemma \ref{FS} we can find $D'\in[\omega]^\omega$ with $FS(D')\subseteq FS(D)$ and such that for each $x\in X\setminus\omega$ there is $k\in\omega$ with $FS(D'\setminus k)\cap f^{-1}[\{x\}]=\emptyset$. Since for each $U\in\mathcal{B}(\infty)$ there are only finitely many $\alpha\in T$ with $\FS(D_\alpha)\notin U$, $\langle f(n)\rangle_{n\in FS(D')}$ IP-converges to $\infty$.

If $f^{-1}[\omega]\notin\Hindman$ then there is $\alpha<\omega_1$ with $(f^{-1}[\omega],f_{\upharpoonright f^{-1}[\omega]})=(H_\alpha,f_\alpha)$. 
We have two subcases: $\alpha\in T$ and $\alpha\notin T$.

Assume $\alpha\in T$.
Using Lemma~\ref{FS}, we find an infinite set $D'\subseteq \omega$ such that 
$\FS(D')\subseteq \FS(D_\alpha)$ and 
for each $k\in\omega$ there is $n\in \omega$ with 
$\FS(D'\setminus n)\cap f^{-1}[k] = \emptyset$.
Since each $U\in\mathcal{B}(\FS(D_\alpha))$ is of the form $\{\FS(D_\alpha)\}\cup f[FS(D_\alpha\setminus k)]\setminus k$ for some $k\in \omega$,
the subsequence 
$\langle f(n)\rangle_{n\in FS(D')}$ IP-converges to $\FS(D_\alpha)$.

Assume $\alpha\notin T$.
Then there is $\beta<\alpha$ such that for each $k\in\omega$ there is $n\in\omega$ with  $f_\alpha[FS(D_\alpha\setminus n)]\subseteq f_\beta[FS(D_\beta\setminus k)]\setminus k$. 
If we take the smallest $\beta<\alpha$ with the above property, then $\beta\in T$
and $\langle f(n)\rangle _{n\in FS(D_\alpha)}$ IP-converges to $\FS(D_\beta)\in X$.

Now we check that $X$ is not an $\I$-space. Define $f:\omega\to X$ by $f(n)=n$ and fix any $B\notin\I$. We claim that $\langle f(n)\rangle _{n\in B}$ does not converge. Clearly, $\langle f(n)\rangle_{n\in B}$ cannot converge to any $x\in\omega$. Moreover, it cannot converge to any $\FS(D_\alpha)$ as 
$U=\{\FS(D_\alpha)\}\cup f_\alpha[ \FS(D_\alpha)]$ is an open neighborhood of $\FS(D_\alpha)$ and $f[B]\setminus U=B\setminus U$ is infinite (since $B\notin\I$ and $f_\alpha[ \FS(D_\alpha)] \in\I$). 

We will show that  $\langle f(n)\rangle _{n\in B}$ cannot converge to $\infty$. Since $B$ is infinite, there is $\alpha<\omega_1$ with 
$H_\alpha=\omega$ and $f_\alpha:H_\alpha\to B$ being the increasing enumeration of $B$.
In particular, $f_\alpha[H_\alpha] = B$.
We have 2 cases: $\alpha\in T$ or $\alpha\notin T$.
 
Suppose first that $\alpha\in T$. Then $f_\alpha[\FS(D_\alpha)]\subseteq B$. 
Since $f_\alpha[\FS(D_\alpha\setminus k)]$ 
is infinite for every $k$, we can inductively take $c_k\in f_\alpha[\FS(D_\alpha\setminus k)] \setminus\{c_i:i<k\}$ to obtain an infinite set $C=\{c_k:k\in\omega\}$
such that 
$C\subseteq B$ and $C\setminus f_\alpha[\FS(D_\alpha\setminus k)]$ is finite for all $k\in\omega$.
Now observe that for each $\beta\in T\setminus\{\alpha\}$ there is $U_\beta\in\mathcal{B}(\FS(D_\beta))$ with $U_\beta\cap C=\emptyset$. Indeed, fix $\beta\in T\setminus\{\alpha\}$. By condition \eqref{eq:W1} there is $k\in\omega$ such that  $f_\alpha[\FS(D_\alpha\setminus k)]\cap f_\beta[\FS(D_\beta\setminus k)]=\emptyset$. 
Let $k_\beta = \max(k,\max(C\setminus f_\alpha[\FS(D_\alpha\setminus k)]))$.
Then $U_\beta=\{\FS(D_\beta)\}\cup f_\beta[\FS(D_\beta\setminus k_\beta)]\setminus k_\beta$ is in $\mathcal{B}(\FS(D_\beta))$ and $U_\beta \cap C=\emptyset$. 
Therefore, $U=\{\infty\}\cup\bigcup_{\beta\in T\setminus \{\alpha\}}U_\beta$ is an open neighborhood of $\infty$ which is disjoint with $C$. Since $C\in[B]^\omega$, this shows that $\langle f(n)\rangle_{n\in B}$ cannot converge to $\infty$.

Suppose now that $\alpha\notin T$. Then there is $\beta<\alpha$ such that for each $k\in\omega$ there is $n\in\omega$ with $f_\alpha[\FS(D_\alpha\setminus n)]\subseteq f_\beta[FS(D_\beta\setminus k)]$. 
If we take the smallest $\beta<\alpha$ with the above property, then $\beta\in T$. 
Then similarly as in the first case, we find an infinite $C\subseteq B$ with $C\setminus f_\beta[FS(D_\beta\setminus k)]$ finite for all $k\in\omega$, and consequently we can find an open neighborhood $U$ of $\infty$ disjoint with $C$.

To finish the entire proof we have to check that limits of convergent sequences and IP-limits of IP-convergent IP-sequences in $X$ are unique. 
By Proposition~\ref{prop:IP-limit-unique-iff-limit-unique} we only have to show the uniqueness of limits. 

Fix a convergent sequence $x_n\in X$. Since $X\setminus\{\infty\}$ is $T_2$ (and limits in $T_2$ space are unique) and $\omega$ is discrete in $X$ (so sequences with limits in $\omega$ have  unique limits) we only need to check that if $\langle x_n\rangle$ converges to some $\FS(D_\alpha)$ then it cannot converge to $\infty$. So suppose that $\langle x_n\rangle$ converges to $A_\alpha$.

If  $x_n= FS(D_\alpha)$ for infinitely many $n$ 
then $U=\{\infty\}\cup\bigcup_{\beta\in T\setminus \{\alpha\}}\{\FS(D_\beta)\}\cup f_\beta[\FS(D_\beta)]$ is an open neighborhood of $\infty$ which omits infinitely many elements of $\langle x_n\rangle$. Hence  $\langle x_n\rangle$ cannot converge to $\infty$.

Suppose now that $x_n \neq  FS(D_\alpha)$ for all but finitely many $n$. Then the set $C=\omega\cap \{x_n:n\in\omega\}$ is infinite.
Since $\langle x_n\rangle$ converges to $FS(D_\alpha)$, $C\setminus f_\alpha[\FS(D_\alpha\setminus k)]$ is finite for any $k\in \omega$.
Then, using condition \eqref{eq:W1}, for each $\beta\in T\setminus\{\alpha\}$ we can find $k_\beta\in\omega$ with $C\cap (f_\beta[\FS(D_\beta\setminus k_\beta)]\setminus k_\beta) = \emptyset$.
Therefore, $U=\{\infty\}\cup\bigcup_{\beta\in T\setminus \{\alpha\}}\{\FS(D_\beta)\}\cup f_\beta[\FS(D_\beta\setminus k_\beta)]\setminus k_\beta$ is an open neighborhood of $\infty$ which is disjoint with $C$. Hence $\langle x_n\rangle$ cannot converge to $\infty$.
\end{proof}

At this point a natural conjecture would be that if $\I\leq_K\Hindman$ then every Hindman space is an $\I$-space. We were able to prove it only under one additional assumption on the ideal $\I$.

An ideal $\I$ is called \emph{P$^+$-ideal} if for every decreasing sequence $A_1,A_2,\dots\notin\I$ there is $A\notin \I$ such that $A\setminus A_n$ is finite for each $n\in \omega$ (\cite[p.~32]{MR2777744}).

\begin{proposition}
	\label{P+}
If $\I\leq_K\Hindman$ and $\I$ is a P$^+$-ideal, then every Hindman space is an $\I$-space.
In particularly, if $\I$ is an $F_\sigma$ ideal and $\I\leq_K\Hindman$ then every Hindman space is an $\I$-space.
\end{proposition}

\begin{proof}
The ``in particular'' part follows
from the first part as it is well known that  each $F_\sigma$ ideal is a P$^+$-ideal 
(see e.g.~\cite[Proposition~5.1]{MR2961261}).

Let $X$ be a Hindman space and fix a sequence $\langle x_n\rangle_{n\in \omega}$ in $X$. Let $f:\omega\to\omega$ be a witness for $\I\leq_K\Hindman$ and define $y_n=x_{f(n)}$ for all $n\in\omega$. Then there is an infinite $D\subseteq\omega$ and $x\in X$ such that for each open neighborhood $U$ of the point $x$ there is $m\in\omega$ with $\{y_n: n\in \FS(D\setminus m)\}\subseteq U$. We have $f[\FS(D\setminus n)]\notin\I$ for each $n\in\omega$. Consider the decreasing sequence of $\I$-positive sets given by $A_i=f[\FS(D\setminus i)]$ for all $i\in\omega$. Since $\I$ is a P$^+$-ideal, there is $B\notin\I$ with $B\setminus A_i$ finite for all $i\in\omega$. We claim that $\langle x_n\rangle_{n\in B}$ converges. Indeed, let $U$ be an open neighborhood of $x$. Then there is $k\in\omega$ with $\{x_n: n\in f[\FS(D\setminus k)]\}=\{y_n: n\in \FS(D\setminus k)\}\subseteq U$. Thus, $\{n\in B: x_n\notin U\}\subseteq B\setminus f[FS(D\setminus k)] = B\setminus A_k$ and the latter set is finite.
\end{proof}

\begin{example}
The family $\Folkman = \{A\subseteq\omega: \exists k\in\omega \,\forall D\in [\omega]^k  \,(\FS(D)\setminus A\neq \emptyset)\}$ is an ideal on $\N$ (it follows from a version of Folkman's theorem see e.g.~\cite[a remark before Theorem~1.3 at page 43]{MR2970862}). 
(This ideal was introduced in \cite{MR2471564}, where the author proved that consistently there is an $\Folkman$-space which is not an $\I_{1/n}$-space. This ideal was also considered in \cite{MR3594409}, where the authors posed a question whether it is homogeneous.)
It is not difficult to show that $\cF$ is $F_\sigma$.
It is obvious that $\Folkman\subseteq \Hindman$, so $\Folkman\leq_K\Hindman$ and each Hindman space is an $\Folkman$-space.
\end{example}

Combining Theorem \ref{THM} and Proposition \ref{P+} we get the following result.

\begin{corollary}
	\label{cor:Katetov-iff-spaces-for-Fsigma}
	If $\I$ is an $F_\sigma$ ideal on $\omega$, then the following conditions are equivalent.
\begin{enumerate}
	\item 	$\I\not\leq_K\Hindman$.
\item The continuum hypothesis implies that there is 
	a Hindman space which is not an $\I$-space.
\item It is consistent that there is 
a Hindman space which is not an $\I$-space.
\end{enumerate}	
\end{corollary}

In \cite{MR2052425} Jones posed a question whether it is consistent that there is a Hindman space which is not van der Waerden. Since the ideal $\cW$ is $F_\sigma$ and $\cW$-spaces coincide with  van der Waerden spaces, the question of Jones is reduced to a purely combinatorial question whether $\cW\not\leq_K\cH$.


\section{A Hindman space which is not an $\I_{1/n}$-space}
\label{sec:Hindman-spaces-which-is-not-summable-space}

From the previous section we know that there are $F_\sigma$ ideals below $\mathcal{H}$ in the Kat\v{e}tov order. At this point one may ask about existence of $F_\sigma$ ideals that are not below $\mathcal{H}$ in the Kat\v{e}tov order. In this section we show that $\I_{1/n}$ is an example of such ideal. By the main theorem from the previous section, this answers in positive a question posed by Jana Fla\v{s}kov\'{a} at the 22nd Summer Conference on Topology and its Applications \cite{Flaskova-slides} about existence of a Hindman space which is not $\I_{1/n}$-space.

\begin{lemma}
\label{lemma:FS-subset-of-very-sparse-set}
	If $D\in[\omega]^\omega$ is very sparse and $E\in [\omega]^\omega$ is such that $\FS(E)\subseteq\FS(D)$, then for every $d\in \omega$ there is $e\in \omega$ such that $\FS(E\setminus e)\subseteq\FS(D\setminus d)$.
\end{lemma}

\begin{proof}
	Suppose that there is $d\in \omega$ such that for every $e\in \omega$ we have  $\FS(E\setminus e)\setminus \FS(D\setminus d)\neq \emptyset$.
	Then we pick inductively numbers $x_n$ such that $x_n\in \FS(E\setminus \{0,1,\dots,x_{n-1}\})\setminus \FS(D\setminus d)$.
	Then $x_n+x_k\in \FS(E)$ for any $n\neq k$.
	Since $x_n\in \FS(E)\subseteq\FS(D)$ and $D$ is sparse, there is a unique set $F_{n}\subseteq D$ with $x_n=\sum F_n$. Since $x_n\notin \FS(D\setminus d)$, we have $F_n\cap d\neq \emptyset$.
	Thus, there is $n\neq k$ with $F_n\cap F_k\neq\emptyset$.
But $D$ is very sparse, so $x_n+x_k\notin \FS(D)$, a contradiction. 
\end{proof}

\begin{theorem}\ 
	\begin{enumerate}
		\item 
$\I_{1/n}\not\leq_K\Hindman$ i.e.~there is no function $f:\omega\to\omega$ such that $f^{-1}[A]\in\Hindman$ for every $A\in \I_{1/n}$.
\item 
Assume the continuum hypothesis. There is a Hindman space which is not $\I_{1/n}$-space.
	\end{enumerate}
\end{theorem}

\begin{proof}
	Since (2) follows from (1) and Theorem~\ref{THM}, below we show only (1).
	
Let $f:\omega\to\omega$ be any function. If there is $n\in\omega$ with $f^{-1}[\{n\}]\notin\Hindman$ then we are done, so we can assume that $f^{-1}[\{n\}]\in\Hindman$ for all $n\in\omega$. 

Let $a_0,a_1,\ldots\in \omega$ be such that
$$\sum_{n\in \omega}\frac{2^{n+1}}{a_n+1}<\infty.$$

Since $f^{-1}[\{0,1,\ldots,a_n\}]\in\Hindman$ for all $n\in\omega$, there is a very sparse infinite $D\subseteq\omega$ such that for each $n\in\omega$ there is $k\in\omega$ with $FS(D\setminus k)\cap f^{-1}[\{0,1,\ldots,a_n\}]=\emptyset$ (by Lemmas \ref{verysparse} and \ref{FS}). 
By shrinking and re-enumerating $D$ if necessary, we can assume that $\FS(D\setminus d_n)\cap f^{-1}[\{0,1,\ldots,a_n\}]=\emptyset$, where $\langle d_i:i\in\omega\rangle$ is the increasing enumeration of $D$. 

We will inductively pick $x_n\in \omega$, $D_n\in [\omega]^\omega$ and auxiliary functions $g_n$ such that for each $n\in\omega$:
\begin{itemize}
\item $x_n<x_{n+1}$, $x_n\in D_n$ and $\FS(D_n\setminus x_n)\cap f^{-1}[\{0,1,\ldots,a_n\}]=\emptyset$;
\item $\FS(D)\supseteq \FS(D_n)\supseteq \FS(D_{n+1})$,
\item $g_n:\FS(\{x_i:i<n\})\to\{0,1\}$,
\end{itemize}
and for each $y\in \FS(\{x_i:i<n\})$:
\begin{itemize}
\item if $g_n(y)=1$ then $(y+\FS(D_n))\cap f^{-1}[\{0,1,\ldots,a_n\}]=\emptyset$ (in particular, $f(y+x_n)>a_n$),
\item if $g_n(y)=0$ and either $g_{n-1}(y)=0$ or $y\in (x_{n-1}+(\FS(\{x_i:i<n-1\})\cap g^{-1}_{n-1}[\{0\}]))$ then $f\restriction (y+\FS(D_n))$ is a constant function whose value is some $c\in f[\FS(\{x_i:i<n\})]$,
\item if $g_n(y)=0$ and either $g_{n-1}(y)=1$ or $y\in (x_{n-1}+(\FS(\{x_i:i<n-1\})\cap g^{-1}_{n-1}[\{1\}]))$ or $y=x_{n-1}$ then $f\restriction (y+FS(D_n))$ is a constant function whose value is some $c>a_{n-1}$.
\end{itemize}

(The initial step for $n=0$)
We start with $D_0=D$, $g_0=\emptyset$ and  any $x_0 \in D_0$. 
Since $d_0\leq x_0$, we have 
$\FS(D_0\setminus x_0)\cap f^{-1}[\{0,1,\ldots,a_0\}] = \emptyset$.

\smallskip

(The initial step for $n=1$)
We have 2 cases.

\emph{Case 1.} 
There is $c\in\omega$ such that $f^{-1}[\{c\}]$ contains $x_0+\FS(C)$ for some infinite $C$ with $\FS(C)\subseteq \FS(D_0\setminus \{0,1,\dots,x_0\})$.
We define $D_1=C$, $g_1(x_0)=0$ and, using Lemma~\ref{lemma:FS-subset-of-very-sparse-set}, pick $x_1\in D_1$
such that 
$\FS(D_1\setminus x_1)\subseteq \FS(D\setminus d_1)$.

Since 
$\FS(D\setminus d_1)\cap f^{-1}[\{0,1,\ldots,a_1\}]=\emptyset$, 
$\FS(D_1\setminus x_1)\cap f^{-1}[\{0,1,\ldots,a_1\}]=\emptyset$.

Since $x_0+\FS(D_1) \subseteq x_0 + \FS(D_0\setminus \{0,1,\dots,x_0\})\subseteq \FS(D_0\setminus x_0)$ and
$\FS(D_0\setminus x_0)\cap f^{-1}[\{0,1,\ldots,a_0\}] = \emptyset$, we obtain $c>a_0$.

\emph{Case 2.} 
For each $c\in\omega$ the set $f^{-1}[\{c\}]$ does not contain any $x_0+\FS(C)$ for infinite $C$ with 
$\FS(C)\subseteq \FS(D_0\setminus \{0,1,\dots,x_0\})$, so using Lemma \ref{FS} we can find an infinite $C$ with $\FS(C)\subseteq \FS(D_0\setminus \{0,1,\dots,x_0\})$ such that for each $m\in\omega$ there is $k_m\in\omega$ with $(x_0+\FS(C\setminus k_m))\cap f^{-1}[\{0,1,\ldots,a_m\}]=\emptyset$. Put $D_1=C\setminus k_1$,
 $g_1(x_0)=1$ and, using Lemma~\ref{lemma:FS-subset-of-very-sparse-set}, pick $x_1\in D_1$
 such that 
 $\FS(D_1\setminus x_1)\subseteq \FS(D\setminus d_1)$.
  
Since 
$\FS(D\setminus d_1)\cap f^{-1}[\{0,1,\ldots,a_1\}]=\emptyset$, 
$\FS(D_1\setminus x_1)\cap f^{-1}[\{0,1,\ldots,a_1\}]=\emptyset$

Since $x_1\in \FS(D_1) \subseteq \FS(D_0\setminus \{0,1,\dots,x_0\})$, we obtain $x_0<x_1$.

Moreover, $(x_0+\FS(D_1))\cap f^{-1}[\{0,1,\ldots,a_1\}]= (x_0+\FS(C\setminus k_1))\cap f^{-1}[\{0,1,\ldots,a_1\}]=\emptyset$.

\smallskip

(The inductive step)
Suppose now that $x_i$, $g_i$ and $D_i$ for all $i<n$ are already defined. 
Let 
\begin{equation*}
	\begin{split}
\{t_0,\ldots,t_{k}\} 
&= 
\FS(\{x_i:i<n\})\setminus \\
& \left(
g_{n-1}^{-1}[\{0\}] 
 \cup(x_{n-1}+(\FS(\{x_i:i<n-1\})\cap g_{n-1}^{-1}[\{0\}]))
\right)
\end{split}
\end{equation*}
and define $g_n(y)=0$ for each $y\in g_{n-1}^{-1}[\{0\}]  \cup  (x_{n-1}+(\FS(\{x_i:i<n-1\})\cap g_{n-1}^{-1}[\{0\}]))$.

We will inductively construct auxiliary sets $E_{-1},E_0,\ldots,E_k\in[\omega]^\omega$ such that $\FS(D_{n-1}\setminus \{0,1,\dots,x_{n-1}\})\supseteq \FS(E_{-1}) \supseteq \dots\supseteq \FS(E_{k})$.

We start with $E_{-1}=D_{n-1}\setminus \{0,1,\dots,x_{n-1}\}$. Suppose now that $i\leq k$ and $E_j$ for all $j<i$ are defined. We have 2 cases.

\emph{Case 1.} 
There is $c\in\omega$ such that $f^{-1}[\{c\}]$ contains $t_i+\FS(C)$ for some infinite $C$ with $\FS(C)\subseteq \FS(E_{i-1})$.
We define $E_i=C$ and $g_n(t_i)=0$.

Note that $c>a_{n-1}$ in this case. Indeed, we have 3 cases:
\begin{enumerate}
	\item $t_i=x_{n-1}$,
	\item $t_i\in x_{n-1}+(\FS(\{x_j:j<n-1\}) \cap g_{n-1}^{-1}[\{1\}])$,
	\item $t_i\in \FS(\{x_j:j<n-1\})$.
\end{enumerate}

In the first case, 
$t_i+\FS(E_i)\subseteq t_i+\FS(D_{n-1}\setminus \{0,1,\dots,x_{n-1}\})$
and
$\FS(D_{n-1}\setminus x_{n-1})\cap f^{-1}[\{0,\ldots,a_{n-1}\}]=\emptyset$, so $c>a_{n-1}$.

In the second case, 
$t_i=y+x_{n-1}$ for some $y\in \FS(\{x_j:j<n-1\})$ with $g_{n-1}(y)=1$. 
Then $t_i+\FS(E_i) \subseteq 
t_i+\FS(D_{n-1}\setminus \{0,1,\dots,x_{n-1}\}) 
=
y+x_{n-1}+\FS(D_{n-1}\setminus \{0,1,\dots,x_{n-1}\}) 
\subseteq 
y+\FS(D_{n-1}\setminus x_{n-1})
$.
Since 
$g_{n-1}(y)=1$, $(y+\FS(D_{n-1}))\cap f^{-1}[\{0,1,\ldots,a_{n-1}\}]=\emptyset$. Thus $c>a_{n-1}$.

In the third case, 
$g_{n-1}(t_i)=1$, 
so $(t_i+\FS(D_{n-1}))\cap f^{-1}[\{0,1,\ldots,a_{n-1}\}]=\emptyset$.
Since 
$t_i+\FS(E_i)\subseteq t_i+\FS(D_{n-1})$, we obtain $c>a_{n-1}$.

\emph{Case 2.} 
For each $c\in\omega$ the set $f^{-1}[\{c\}]$ does not contain any $t_i+\FS(C)$ for infinite $C$ with $\FS(C)\subseteq \FS(E_{i-1})$, so using Lemma~\ref{FS} we can find an infinite $C$ with $\FS(C)\subseteq \FS(E_{i-1})$ such that for each $m\in\omega$ there is $k_m\in\omega$ with $(t_i+\FS(C\setminus k_m))\cap f^{-1}[\{0,1,\ldots,a_m\}]=\emptyset$. Put $E_i=C\setminus k_n$ and $g_n(t_i)=1$. 


Once all $E_i$ for $i\leq k$ are constructed, put $D_n=E_k$ and, using Lemma~\ref{lemma:FS-subset-of-very-sparse-set}, pick $x_n\in D_n$
such that 
$\FS(D_n\setminus x_n)\subseteq \FS(D\setminus d_n)$.

This finishes the inductive construction of the sequences $\langle x_n \rangle$, $\langle D_n \rangle$ and $\langle g_n\rangle$.

\smallskip

Define $X=\{x_i:i\in\omega\}$. Obviously, $\FS(X)\notin\Hindman$. To end the proof we need to check that $f[\FS(X)]\in\I_{1/n}$ i.e.~we have to show that
$$
\sum_{i\in f[\FS(X)]}\frac{1}{i+1} = 
\sum_{i\in f[\FS(X_0)]}\frac{1}{i+1} 
+
\sum_{n=1}^\infty
\left(\sum_{i\in f[\FS(X_n)]\setminus f[\FS(X_{n-1})]}\frac{1}{i+1} \right)<\infty,
$$
where 
$X_n=\{x_0,x_1,\ldots,x_n\}$ for every $n\in\omega$.

Let $n\geq 1$. If $i\in f[\FS(X_n)] \setminus f[\FS(X_{n-1})]$, then either $i=f(x_n)$ or $i = f(y+x_n)$ for some $y\in \FS(X_{n-1})$.
It follows from the properties of the sequence $\langle x_n\rangle$ that $f(x_n)>a_n$ (hence $1/(i+1)<1/(a_n+1)$) and for the value $f(y+x_n)$ we have three cases.
\begin{itemize}
	\item If $g_n(y)=1$ then $f(y+x_n)>a_n$, so $1/(i+1)<1/(a_n+1)$.
	\item If $g_n(y)=0$ and 
	either $g_{n-1}(y)=0$ or $y\in (x_{n-1}+(\FS(\{x_i:i<n-1\})\cap g^{-1}_{n-1}[\{0\}]))$, 
	then $f\restriction (y+\FS(D_n))$ is a constant function whose value is some $c\in f[\FS(\{x_i:i<n\})]$. Note that this case is not possible here because $i\notin f[\FS(X_{n-1})]$.
	
	\item If $g_n(y)=0$ and 
	 either $g_{n-1}(y)=1$ or $y\in (x_{n-1}+(\FS(\{x_i:i<n-1\})\cap g^{-1}_{n-1}[\{1\}]))$ or $y=x_{n-1}$,
	 then $f\restriction (y+FS(D_n))$ is a constant function whose value is some $c>a_{n-1}$, so $1/(i+1)<1/(a_{n-1}+1)$.
\end{itemize}
Thus,
\begin{equation*}
	\begin{split}
		\sum_{i\in f[\FS(X_n)]\setminus f[\FS(X_{n-1})]}\frac{1}{i+1} 
& \leq 
\frac{1}{a_n+1} + \frac{|\FS(X_{n-1})|}{a_{n-1}+1}
\leq 
\frac{2^{n}}{a_{n-1}+1}.
\end{split}
\end{equation*}
All in all, 
$$
\sum_{i\in f[\FS(X)]}\frac{1}{i+1} \leq 
\sum_{i\in f[\FS(X_0)]}\frac{1}{i+1}  + \sum_{n=1}^\infty \frac{2^{n}}{a_{n-1}+1}
=
\frac{1}{f(x_0)+1}  + \sum_{n=1}^\infty \frac{2^{n}}{a_{n-1}+1}<\infty.$$
	This finishes the proof.
\end{proof}


\bibliographystyle{amsplain}
\bibliography{paper}

\end{document}